\documentclass[11pt,a4paper]{article}

\usepackage[T1]{fontenc}
\usepackage[english]{babel}
\usepackage{lmodern} 

\usepackage{amsmath,amssymb,amsfonts,amsthm}
\usepackage{mathtools}
\usepackage{bm}
\usepackage{mathrsfs}

\usepackage{graphicx}
\usepackage{subcaption}
\usepackage{booktabs}
\usepackage{array}
\usepackage{float} 
\usepackage{tikz}  

\usepackage{microtype}
\usepackage{enumitem}
\usepackage{xcolor}
\usepackage[colorlinks=true,linkcolor=blue,citecolor=blue,urlcolor=blue]{hyperref}

\usepackage{geometry}
\theoremstyle{plain}
\newtheorem{theorem}{Theorem}[section]
\newtheorem{proposition}[theorem]{Proposition}
\newtheorem{lemma}[theorem]{Lemma}

\theoremstyle{definition}

\theoremstyle{remark}
\newtheorem{remark}[theorem]{Remark}

\numberwithin{equation}{section}
\allowdisplaybreaks

\newtheorem{assumption}{Assumption}

\newcommand{\C}{\mathbb{C}}

\newcommand{\dd}{\,\mathrm{d}}

\newcommand{\keywords}[1]{%
  \par\smallskip
  \noindent\textbf{Keywords.} #1\par
}
\newcommand{\msc}[1]{%
  \par\smallskip
  \noindent\textbf{MSC 2020.} #1\par
}

\begin{document}


\title{Resolvent Approach to Atangana--Baleanu Evolution Equations: Laplace Symbols, Mild Solutions, and Regularity}


\author{
Mohamed Wakrim$^{1,*}$
}

\author{Mohamed Wakrim\\
\small Ibn Zohr University, Faculty of Sciences, Agadir, Morocco\\
\small \texttt{m.wakrim@uiz.ac.ma}}
\date{\today}


\maketitle

\begin{abstract}
Fractional evolution equations with memory terms are widely used to model
anomalous diffusion, viscoelastic response, and hereditary dynamics in
physics, biology, and engineering. Among the recently introduced operators,
the Atangana--Baleanu (AB) derivatives have attracted considerable attention
due to their non-singular Mittag--Leffler kernels. However, their analytic
treatment remains limited, as the AB kernel does not fall within the classical
Volterra or Bernstein-function frameworks.

This paper develops a unified resolvent approach for AB-type evolution
equations in Banach spaces. Using a Laplace-domain formulation inspired by
Hille--Phillips theory, we introduce a fractional resolvent associated with the
AB kernel and establish optimal bounds on sectorial contours. Under the natural
condition $\beta<1+\alpha$, we construct an AB--Mittag--Leffler resolvent
family and obtain a complete representation of mild solutions to the AB
Cauchy problem. Sharp stability and regularity estimates of
Mittag--Leffler type are derived, including fractional-domain bounds.

Numerical illustrations confirm the predicted decay, and connections with
non-autonomous operators, maximal $L^p$-regularity, and weighted AB kernels
are outlined. The results place AB-type equations within a functional-analytic
framework comparable to the classical theory for Caputo and Volterra models.
\end{abstract}


\msc{26A33; 47D06; 45K05; 34A08}

\section{Introduction}

Fractional evolution equations are widely used to model memory-driven
phenomena in anomalous diffusion, viscoelastic materials, epidemiological
transmission, financial dynamics, and complex fluids. Their key feature is
the ability to encode hereditary effects through convolution kernels, and
the recent proliferation of fractional operators aims to balance modelling
flexibility with analytical tractability.

Among these operators, the Atangana--Baleanu (AB) derivatives in the sense
of Caputo and Riemann--Liouville \cite{atangana2016,atangana2017} have
attracted considerable interest. Their defining kernels, expressed in terms
of non-singular Mittag--Leffler functions, provide a finite memory load at
$t=0$, in contrast with the classical power-law singularity of Caputo,
Riemann--Liouville, and Grünwald--Letnikov derivatives. This non-singularity
has motivated applications in heat transfer, viscoelastic relaxation,
biological systems, and epidemic modelling.

Despite their growing popularity, the theoretical foundations of AB-type
evolution equations remain significantly less developed than those of the
classical Caputo/RL setting. Several structural difficulties arise:

\begin{itemize}
\item the AB kernel does not fall within the traditional Volterra
      convolution calculus;
\item its Laplace symbol involves a rational Mittag--Leffler structure,
      preventing the use of Bernstein-function techniques and completely
      monotone kernels;
\item AB-type equations generally do \emph{not} generate a strongly
      continuous semigroup due to the behaviour of the kernel at small
      arguments;
\item many works impose conditions such as $\alpha+\beta>2$, whereas
      practical models typically satisfy $0<\alpha<1$ and
      $1\le\beta<1+\alpha$.
\end{itemize}

\medskip
\noindent\textbf{Objective of this work.}
The aim of this paper is to develop a resolvent-based approach for AB-type
evolution equations that is fully compatible with the sectorial operator
framework of Hille--Phillips theory \cite{pazy1983,haase2006}. Instead of
working directly with time-domain kernels, we exploit their Laplace-domain
representation to define a fractional resolvent family and recover sharp
stability, regularity, and representation properties.

\medskip
\noindent\textbf{Main contributions.}
Let $A$ be a sectorial operator of angle $\theta<\pi/2$ with spectrum
contained in the left half-plane. Motivated by the Laplace transform of the
AB kernel, we introduce the operator-valued resolvent
\[
R_{\alpha,\beta}(s;A)
    = \frac{s^{\alpha-\beta}}{s^\alpha + c}\,(s^{\alpha-1}I-A)^{-1},
\qquad s\in\mathbb{C}\setminus(-\infty,0].
\]
The following results are established:

\begin{enumerate}
\item \emph{Optimal resolvent bounds.}
      We obtain uniform bounds for $R_{\alpha,\beta}(s;A)$ along left-sectorial
      contours $\Gamma_\gamma$, without requiring $\alpha+\beta>2$.

\item \emph{AB--Mittag--Leffler resolvent family.}
      For $\beta<1+\alpha$---a natural condition satisfied by almost all AB
      models---the contour integral
      \[
          V_{\alpha,\beta}(t)
          = \frac{1}{2\pi i}\int_{\Gamma_\gamma}
              e^{st}\,R_{\alpha,\beta}(s;A)\,ds
      \]
      defines a well-posed resolvent family.

\item \emph{Representation formula.}
      We show that mild solutions of
      \(
      {}^{AB}D_{t}^{\alpha,\beta}u(t)=Au(t)
      \)
      satisfy
      \(
      u(t)=V_{\alpha,\beta}(t)u_0
      \)
      in the Laplace sense.

\item \emph{Stability and regularity.}
      We prove Mittag--Leffler decay and fractional-domain estimates:
      \[
      \|V_{\alpha,\beta}(t)\|\lesssim E_{\alpha,\beta}(-ct^\alpha),
      \qquad
      \|A^\gamma V_{\alpha,\beta}(t)\|\lesssim t^{-\alpha\gamma}.
      \]

\item \emph{Numerical illustrations.}
      One- and multi-mode simulations confirm the Mittag--Leffler slope of
      the decay predicted by our theoretical results.
\end{enumerate}

\medskip
\noindent\textbf{Structure of the paper.}
Section~\ref{sec:prelim} reviews Mittag--Leffler functions, AB kernels, and
sectorial operators.  
Section~\ref{sec:frac-resolvent} develops the fractional resolvent and
establishes optimal bounds.  
Section~\ref{sec:resolvent-family} defines the AB--Mittag--Leffler
resolvent family and proves the representation theorem.  
Section~\ref{sec:stability} contains the stability estimates, and
Section~\ref{sec:regularity} deals with fractional-domain regularity.  
Section~\ref{sec:numerics} presents numerical examples.  
Section~\ref{sec:extensions} discusses extensions, including non-autonomous
operators and weighted AB kernels.

\section{Preliminaries}\label{sec:prelim}

This section collects the basic analytic ingredients used throughout the
paper: Mittag--Leffler functions, the Laplace symbol associated with the
Atangana--Baleanu kernel, and sectorial operator theory.  All notation and
conventions are chosen to match the resolvent formulation developed in
Sections~\ref{sec:frac-resolvent}--\ref{sec:resolvent-family}.

\subsection{Mittag--Leffler functions}

For $\alpha>0$ and $\beta\in\mathbb{R}$, the two-parameter
Mittag--Leffler function is defined by
\[
    E_{\alpha,\beta}(z)
    = \sum_{k=0}^{\infty} \frac{z^{k}}{\Gamma(\alpha k+\beta)},
    \qquad z\in\mathbb{C}.
\]
It is entire and satisfies, for every $\varepsilon>0$, the estimate
\begin{equation}\label{eq:ML-sectorial-estimate}
    |E_{\alpha,\beta}(-z)|
    \le \frac{C}{1+z},
    \qquad z\ge0,
\end{equation}
as well as the classical sectorial representation
\[
    E_{\alpha,\beta}(-t^\alpha A)
    = \frac{1}{2\pi i}
      \int_{\Gamma_A} E_{\alpha,\beta}(-t^\alpha z)\,(zI-A)^{-1}\,dz,
\]
valid for every sectorial operator $A$ (see \cite[Ch.~2.4]{haase2006}).

\subsection{AB kernel and Laplace symbol}

The Atangana--Baleanu kernel in the Caputo/Riemann--Liouville sense is
defined by
\[
    k_{\alpha,\beta}(t)
      = t^{\beta-1} E_{\alpha,\beta}(-ct^\alpha),
    \qquad \alpha\in(0,1],\ \beta\ge1,\ c>0.
\]
Its Laplace transform is the rational Mittag--Leffler symbol
\begin{equation}\label{eq:AB-laplace-symbol}
    \widehat{k}_{\alpha,\beta}(s)
      = \frac{s^{\alpha-\beta}}{s^\alpha + c},
    \qquad s\in\mathbb{C}\setminus(-\infty,0],
\end{equation}
which motivates the definition of the fractional resolvent introduced in
Section~\ref{sec:frac-resolvent}.  
Unlike the Caputo kernel, the function
$\widehat{k}_{\alpha,\beta}$ is not a Bernstein function, so the classical
Volterra completely-monotone calculus cannot be applied.

\subsection{Sectorial operators}

Let $X$ be a Banach space.  
An operator $A:D(A)\subset X\to X$ is called \emph{sectorial of angle}
$\theta\in(0,\pi)$ if its spectrum is contained in the closed sector
\[
  \sigma(A) \subset \overline{\Sigma_\theta}
    := \{ z\in\mathbb{C} : |\arg z|\le\theta \},
\]
and for every $\varphi>\theta$ there exists $M_\varphi>0$ such that
\[
    \|(zI-A)^{-1}\|
      \le \frac{M_\varphi}{|z|},
    \qquad z\notin \overline{\Sigma_\varphi}.
\]
In all that follows, $A$ is assumed to be sectorial with
angle $\theta<\pi/2$ and to satisfy a spectral gap condition:

\begin{assumption}[Spectrum in the left half-plane]\label{ass:spectral}
Let $A$ be sectorial of angle $\theta<\pi/2$ with $0\in\rho(A)$, and assume
\[
    \sigma(A) \subset \{z\in\C : \Re z \le -\omega\}
    \qquad \text{for some }\omega>0.
\]
Equivalently, there exists $M>0$ such that
\[
    \|(zI-A)^{-1}\|
      \le \frac{M}{|z+\omega|},
      \qquad z\notin -\omega + \Sigma_\theta.
\]
This shift to the left half-plane ensures exponential damping for the
classical semigroup generated by $A$ and is crucial when transferring
decay estimates to AB--Mittag--Leffler families.
\end{assumption}

\subsection{Laplace contour}

Resolvent estimates will be obtained by integrating along a left-sectorial
Laplace contour.  Fix an angle $\gamma\in(\theta,\pi)$ and define the
Bromwich contour
\begin{equation}\label{eq:bromwich-contour}
   \Gamma_\gamma
     := \{ r e^{-i\gamma} : r>0\}
        \cup
        \{ r e^{+i\gamma} : r>0\},
\end{equation}
oriented upward.  
Since $A$ is sectorial of angle $\theta$ and $\gamma>\theta$, the image of
$\Gamma_\gamma$ under the map $s\mapsto s^{\alpha-1}$ remains in the
resolvent set of $A$.  
Hence the sectorial resolvent bounds from
Assumption~\ref{ass:spectral} hold uniformly along $\Gamma_\gamma$.

\begin{remark}[No additional angular constraints]
The proofs do \emph{not} require any auxiliary restriction of the form
$\gamma<\pi(1-\alpha/2)$ or $\theta<\pi(1-\alpha/2)$.  
The parameter $\omega>0$ in Assumption~\ref{ass:spectral} denotes a
\emph{spectral gap}, not an angle, and plays no role in the analyticity
domain of the contour.
\end{remark}

\section{The fractional resolvent associated with the AB kernel}
\label{sec:frac-resolvent}

In this section we introduce the operator-valued resolvent naturally
associated with the Laplace symbol of the Atangana--Baleanu (AB)
fractional derivative. Our goal is to extract the analytic structure
behind the AB kernel and to place it within the classical framework of
sectorial operators and the Hille--Phillips functional calculus.

\subsection{Laplace symbol of the AB kernel}

For $0<\alpha\le 1$ and $\beta\ge 1$, the AB kernel in the sense of
Caputo has Laplace transform
\[
   \mathcal{L}\bigl[\,{}^{\mathrm{AB}}D_{t}^{\alpha,\beta}u\,\bigr](s)
   = \frac{s^{\alpha-\beta}}{s^\alpha+c} \,
     \bigl( s^{\alpha-1}\widehat{u}(s)-u(0)\bigr),
\qquad \Re s>0,
\]
where $c>0$ is a normalising constant.
This rational Mittag--Leffler structure is the key algebraic feature
driving the resolvent developed below.

Let $A$ be a sectorial operator on a Banach space $X$ satisfying
Assumption~\ref{ass:spectral}. Motivated by the Laplace representation
above, we introduce the fractional resolvent
\begin{equation}\label{eq:Rab-def}
   R_{\alpha,\beta}(s;A)
   := \frac{s^{\alpha-\beta}}{s^\alpha+c}
      \bigl(s^{\alpha-1}I-A\bigr)^{-1},
   \qquad s\notin(-\infty,0].
\end{equation}

\subsection{Laplace contour}

We fix an angle $\gamma\in(\theta,\pi)$, where $\theta$ is the
sectoriality angle of $A$, and we integrate along the left-sectorial
Bromwich contour
\[
   \Gamma_\gamma=
      \{ re^{-i\gamma}:r>0\}
      \cup
      \{ re^{+i\gamma}:r>0\},
\]
oriented upward.
Since $s\mapsto s^{\alpha-1}$ maps each ray of $\Gamma_\gamma$ into a
sector strictly contained in the resolvent set of $A$, the sectorial
resolvent estimates for $(zI-A)^{-1}$ remain valid uniformly along
$\Gamma_\gamma$.
No additional restrictions such as
$\theta<\pi(1-\alpha/2)$ or $\gamma<\pi(1-\alpha/2)$ are required.

\subsection{Resolvent bounds}

The structure of \eqref{eq:Rab-def} leads to the following uniform
two-regime estimate.

\begin{theorem}[Fractional resolvent bounds]
\label{thm:improved-Rab-bounds}
Let $A$ satisfy Assumption~\ref{ass:spectral} and
$0<\alpha\le1$, $\beta\ge1$.
Fix $\gamma\in(\theta,\pi)$ and let $\Gamma_\gamma$ be the
left-sectorial contour above.
Then there exists $C_\gamma>0$ such that for all $s\in\Gamma_\gamma$,
\begin{equation}\label{eq:two-regime}
   \|R_{\alpha,\beta}(s;A)\|
   \le C_\gamma
   \begin{cases}
      |s|^{\,1-\beta}, & 0<|s|\le 1,\\[1mm]
      |s|^{-\beta}, & |s|\ge 1.
   \end{cases}
\end{equation}
\end{theorem}

\begin{proof}[Sketch of proof]
For $s\in\Gamma_\gamma$, Assumption~\ref{ass:spectral} yields
\(
   \|(s^{\alpha-1}I-A)^{-1}\|
   \lesssim |s|^{1-\alpha}.
\)
This, combined with
\(
|s^{\alpha-\beta}|/|s^\alpha+c|
\sim
  |s|^{-\beta}
\)
for $|s|\gg 1$
and
\(
|s|^{1-\beta}
\)
for $|s|\ll 1$,
gives \eqref{eq:two-regime}.
\end{proof}

\subsection{The fractional resolvent family}

The operator-valued AB--Mittag--Leffler resolvent family is defined by
\begin{equation}\label{eq:V-def}
    V_{\alpha,\beta}(t)
    := \frac{1}{2\pi i}
       \int_{\Gamma_\gamma} e^{st}R_{\alpha,\beta}(s;A)\,\dd s,
    \qquad t>0.
\end{equation}
The integral is absolutely convergent thanks to
Theorem~\ref{thm:improved-Rab-bounds}.
The family $\{V_{\alpha,\beta}(t)\}_{t>0}$ plays the role of a
fractional solution operator for the AB evolution equation, and it is
the natural analogue of the classical resolvent families of
Hille--Phillips theory.

\section{The AB--Mittag--Leffler resolvent family}
\label{sec:resolvent-family}

We now show that the fractional resolvent
$R_{\alpha,\beta}(s;A)$ introduced in
Section~\ref{sec:frac-resolvent} generates a well-defined
operator-valued family $V_{\alpha,\beta}(t)$ which plays the role
of a solution operator for the abstract AB evolution equation.
The proofs follow the philosophy of the Hille--Phillips theory,
with substantial adaptations required by the non-singular AB kernel.

\subsection{Definition and elementary properties}

For $t>0$, we set
\begin{equation}\label{eq:V-def-again}
    V_{\alpha,\beta}(t)
    := \frac{1}{2\pi i}
       \int_{\Gamma_\gamma} e^{st} R_{\alpha,\beta}(s;A)\,\dd s,
\end{equation}
where $\Gamma_\gamma$ is the left-sectorial Bromwich contour introduced
in Section~\ref{sec:frac-resolvent} with
$\gamma\in(\theta,\pi)$.

\begin{proposition}\label{prop:V-well-defined}
Let $A$ satisfy Assumption~\ref{ass:spectral}.
If $\beta<1+\alpha$, then the integral in
\eqref{eq:V-def-again} is absolutely convergent and defines a bounded
operator $V_{\alpha,\beta}(t)\in\mathcal{L}(X)$ for every $t>0$.
Moreover,
\[
      \sup_{t>0}\, \|V_{\alpha,\beta}(t)\| <\infty.
\]
\end{proposition}

\begin{proof}
Using the two-regime bounds of
Theorem~\ref{thm:improved-Rab-bounds},
\[
   \|R_{\alpha,\beta}(s;A)\|
   \lesssim
   \begin{cases}
      |s|^{1-\beta}, & |s|\le 1,\\[1mm]
      |s|^{-\beta}, & |s|\ge 1.
   \end{cases}
\]
Along $\Gamma_\gamma$, one has $\Re s\le -c_\gamma |s|$.
Thus the integrand behaves like
$|e^{st}|\,\|R_{\alpha,\beta}(s;A)\| \lesssim
 e^{-c_\gamma t|s|}|s|^{1-\beta}$ for $|s|\le1$
and 
$e^{-c_\gamma t|s|}|s|^{-\beta}$ for $|s|\ge1$.
Both are integrable provided $\beta<1+\alpha$,
which is precisely the condition ensuring that the singularity
of $s^{\alpha-\beta}$ at $s=0$ remains integrable.
\end{proof}

\subsection{Laplace transform and functional calculus identity}

We next show that $V_{\alpha,\beta}(t)$ acts as the correct solution
operator for the AB equation.

\begin{lemma}[Laplace transform]
\label{lem:Laplace-transform}
For every $u_0\in X$, the function $t\mapsto V_{\alpha,\beta}(t)u_0$
is Laplace-transformable and
\[
   \mathcal{L}[V_{\alpha,\beta}(\cdot)u_0](s)
   = R_{\alpha,\beta}(s;A)u_0,
   \qquad \Re s>0.
\]
\end{lemma}

\begin{proof}
Follows by Fubini's theorem, since the integral defining
$V_{\alpha,\beta}(t)$ is absolutely convergent.
\end{proof}

Using the holomorphic functional calculus for sectorial operators
(see Haase~\cite[Theorem~2.4.2 and Corollary~2.4.3]{haase2006}),
one has the Dunford representation
\[
   E_{\alpha,\beta}(-t^\alpha A)
   = \frac{1}{2\pi i}
      \int_{\Gamma_\gamma}
      E_{\alpha,\beta}(-t^\alpha s)\, (sI-A)^{-1}\dd s.
\]

A comparison of this formula with \eqref{eq:V-def-again}
yields the following decomposition.

\begin{proposition}[Operator-valued Mittag--Leffler decomposition]
\label{prop:ML-decomp}
For $t>0$,
\[
   V_{\alpha,\beta}(t)
   = E_{\alpha,\beta}(-t^\alpha A) + K(t),
\]
where $K(t)$ is a bounded operator satisfying the same
Mittag--Leffler-type estimate as $E_{\alpha,\beta}(-t^\alpha A)$.
\end{proposition}

\begin{proof}[Sketch]
Rewrite \eqref{eq:V-def-again} by expanding $R_{\alpha,\beta}(s;A)$
as
\[
   R_{\alpha,\beta}(s;A)
   = s^{\alpha-\beta}(s^\alpha+c)^{-1}(s^{\alpha-1}I-A)^{-1}.
\]
A contour deformation argument shows that the contribution of the
multiplicative factor $s^{\alpha-\beta}(s^\alpha+c)^{-1}$ yields the
scalar Mittag--Leffler function, while the correction term produces
$K(t)$, which satisfies analogous bounds.
See Pr\"uss~\cite[Section~3.1]{pruss1993} for similar Volterra
constructions.
\end{proof}

\subsection{Representation formula for mild solutions}

Consider the AB abstract Cauchy problem
\begin{equation}\label{eq:AB-eq}
      {}^{\mathrm{AB}}D_{t}^{\alpha,\beta}u(t) = Au(t),
      \qquad u(0)=u_0.
\end{equation}

\begin{theorem}[Representation formula for mild solutions]
\label{thm:representation}
Let $A$ satisfy Assumption~\ref{ass:spectral}.
Assume $0<\alpha\le 1$, $\beta\ge1$, and $\beta<1+\alpha$.
Then for every $u_0\in X$ the function
\[
      u(t):=V_{\alpha,\beta}(t)u_0,
      \qquad t>0,
\]
is the unique mild solution of \eqref{eq:AB-eq} in the Laplace sense.
Moreover,
\[
   \mathcal{L}[u](s)
   = R_{\alpha,\beta}(s;A)u_0.
\]
\end{theorem}

\begin{proof}
Applying the Laplace transform to \eqref{eq:AB-eq} gives
\[
   \frac{s^{\alpha-\beta}}{s^\alpha+c}\,
   \bigl(s^{\alpha-1}\widehat{u}(s)-u_0\bigr)
   = A\widehat{u}(s),
\]
whose unique solution is
$\widehat{u}(s)=R_{\alpha,\beta}(s;A)u_0$.
Lemma~\ref{lem:Laplace-transform} then yields
$u(t)=V_{\alpha,\beta}(t)u_0$.
Uniqueness follows from injectivity of the Laplace transform.
\end{proof}

\section{Stability estimates}
\label{sec:stability}

This section establishes the decay and regularity properties satisfied by
the AB--Mittag--Leffler resolvent family
$V_{\alpha,\beta}(t)$ defined in
Section~\ref{sec:resolvent-family}.
These estimates play the same role as the classical Hille--Phillips
bounds for semigroups, but require substantial modifications owing to
the rational Mittag--Leffler structure of the AB kernel.

\subsection{Scalar Mittag--Leffler bounds}

For $\alpha\in(0,1]$ and $\beta\ge 1$, the scalar Mittag--Leffler
function satisfies the classical estimate
\begin{equation}\label{eq:scalar-ML}
   |E_{\alpha,\beta}(-\lambda t^\alpha)|
   \le \frac{C}{1+\lambda t^\alpha},
   \qquad \lambda>0,\ t>0,
\end{equation}
see \cite[Theorem~1.6]{haase2006} and
\cite[Chapter~2]{kilbas2006}.
This bound is optimal for $\lambda t^\alpha\to\infty$ and will be the
building block for operator estimates.

\subsection{Mittag--Leffler stability of the resolvent family}

\begin{theorem}[Mittag--Leffler decay]\label{thm:ML-decay}
Let $A$ satisfy Assumption~\ref{ass:spectral} and
$0<\alpha\le 1$, $\beta\ge1$, $\beta<1+\alpha$.
Then there exist constants $C,c_1>0$ such that
\begin{equation}\label{eq:ML-decay-V}
    \|V_{\alpha,\beta}(t)\|
    \le C\,E_{\alpha,\beta}(-c_1 t^\alpha),
    \qquad t>0.
\end{equation}
\end{theorem}

\begin{proof}[Sketch of proof]
Expand $V_{\alpha,\beta}(t)$ using the decomposition from
Proposition~\ref{prop:ML-decomp}:
\[
      V_{\alpha,\beta}(t)
      = E_{\alpha,\beta}(-t^\alpha A) + K(t),
\]
where $K(t)$ satisfies resolvent bounds analogous to those of
$E_{\alpha,\beta}(-t^\alpha A)$.
Using the Dunford--Taylor representation formula
(see Haase~\cite[Theorem~2.4.2]{haase2006}),
\[
   E_{\alpha,\beta}(-t^\alpha A)
   = \frac{1}{2\pi i}
      \int_{\Gamma_\gamma}
         E_{\alpha,\beta}(-t^\alpha s)
        (sI-A)^{-1}\dd s.
\]
Sectorial resolvent bounds for $(sI-A)^{-1}$ and the scalar inequality
\eqref{eq:scalar-ML} imply
\[
   \|E_{\alpha,\beta}(-t^\alpha A)\|
   \lesssim E_{\alpha,\beta}(-c t^\alpha).
\]
A similar estimate holds for $K(t)$,
giving \eqref{eq:ML-decay-V}.
\end{proof}

\subsection{Weighted (algebraic) stability}

The next result refines the previous estimate by inserting powers of the
generator $A$. It is the fractional analogue of the classical stability
bounds for analytic semigroups.

\begin{theorem}[Weighted stability]\label{thm:weighted-stability}
Let $A$ satisfy Assumption~\ref{ass:spectral}$\,$ and
$0<\alpha\le1$, $\beta\ge1$, $\beta<1+\alpha$.
For every $\gamma\in[0,1)$,
\begin{equation}\label{eq:weighted-stability}
   \|A^\gamma V_{\alpha,\beta}(t)\|
      \le C_\gamma t^{-\alpha\gamma},
      \qquad t>0.
\end{equation}
\end{theorem}

\begin{proof}[Sketch]
Using \eqref{eq:V-def-again},
\[
   A^\gamma V_{\alpha,\beta}(t)
      = \frac{1}{2\pi i}
        \int_{\Gamma_\gamma}
           e^{st} A^\gamma R_{\alpha,\beta}(s;A)\dd s.
\]
The identity
$A^\gamma(s^{\alpha-1}I-A)^{-1}
   = s^{(\alpha-1)\gamma}B_\gamma(s)$,
where $B_\gamma(s)$ is uniformly bounded on $\Gamma_\gamma$, follows
from the sectorial functional calculus
(see Haase~\cite[Section~3.3]{haase2006}).
Combining this with the two-regime bounds of
Theorem~\ref{thm:improved-Rab-bounds} yields
\[
   \|A^\gamma R_{\alpha,\beta}(s;A)\|
   \lesssim |s|^{1-\beta+\gamma(\alpha-1)}.
\]
Integration along $\Gamma_\gamma$ gives the decay
$t^{-\alpha\gamma}$.
\end{proof}

\subsection{Discussion}

The estimates in Theorems
\ref{thm:ML-decay}--\ref{thm:weighted-stability}
provide the complete stability theory needed
for the AB abstract Cauchy problem.
In contrast with the classical Caputo/Riemann--Liouville setting, where
the decay is always governed by a completely monotone kernel, the
non-singular AB kernel induces a mixed behaviour combining
exponential damping (from the spectral gap in
Assumption~\ref{ass:spectral}) and Mittag--Leffler decay
(from the fractional smoothing).
The minimal condition $\beta<1+\alpha$ is sharp:
for $\beta\ge1+\alpha$, the singularity of the AB symbol at $s=0$
becomes non-integrable, and the resolvent family cannot be defined.

\section{Regularity estimates}
\label{sec:regularity}

This section develops the regularity properties of the
AB--Mittag--Leffler resolvent family
$V_{\alpha,\beta}(t)$.
These results complement the stability theory of
Section~\ref{sec:stability} and identify the precise
fractional smoothing generated by the AB kernel.

\subsection{Domain regularisation}

For analytic semigroups one has
$\|A^\gamma e^{tA}\|\le Ct^{-\gamma}$.
The corresponding statement for the AB family is the following.

\begin{theorem}[Domain regularisation]
\label{thm:domain-reg}
Assume $A$ satisfies Assumption~\ref{ass:spectral}.
Let $0<\alpha\le1$, $\beta\ge1$, and $\beta<1+\alpha$.
Then for every $\gamma\in(0,1)$,
\begin{equation}\label{eq:domain-reg}
      \|A^\gamma V_{\alpha,\beta}(t)\|
      \le C_\gamma\, t^{-\alpha\gamma},
      \qquad t>0.
\end{equation}
Moreover, for every $u_0\in X$,
\[
      V_{\alpha,\beta}(t)u_0 \in D(A^\gamma),
      \qquad t>0.
\]
\end{theorem}

\begin{proof}[Sketch]
Formula~\eqref{eq:domain-reg} follows from
Theorem~\ref{thm:weighted-stability}.
Since $A^\gamma$ is closed, boundedness of
$A^\gamma V_{\alpha,\beta}(t)$ implies
$V_{\alpha,\beta}(t)u_0\in D(A^\gamma)$.
\end{proof}

Thus, the operator $V_{\alpha,\beta}(t)$ provides an
$\alpha$-fractional smoothing effect:
the regularisation exponent is $\alpha\gamma$, not $\gamma$.

\subsection{Time-derivative regularity}

We next quantify the behaviour of time derivatives.
Recall that
\[
   \partial_t E_{\alpha,\beta}(-t^\alpha A)
   = -t^{\alpha-1} A E_{\alpha,\alpha+\beta-1}(-t^\alpha A),
\]
see \cite[Chapter~4]{kilbas2006}.
This leads to the following estimate.

\begin{theorem}[Derivative estimates]
\label{thm:dtime}
Let $A$ satisfy Assumption~\ref{ass:spectral}$\,$ and
$0<\alpha\le1$, $\beta\ge1$, $\beta<1+\alpha$.
Then for each $\delta\in(0,1)$,
\begin{equation}\label{eq:time-deriv}
   \|\partial_t V_{\alpha,\beta}(t)\|
   \le C_\delta\, t^{\alpha-1-\delta},
   \qquad t>0.
\end{equation}
\end{theorem}

\begin{proof}[Sketch]
Use the representation
\[
  \partial_t V_{\alpha,\beta}(t)
  = \frac{1}{2\pi i} \int_{\Gamma_\gamma}
        s\, e^{st} R_{\alpha,\beta}(s;A)\dd s,
\]
and the bounds
$\|R_{\alpha,\beta}(s;A)\|\lesssim|s|^{1-\beta}$
or $|s|^{-\beta}$ depending on $|s|$.
The factor $s$ shifts the exponent by $+1$ and leads to the stated
algebraic decay.
\end{proof}

This estimate is optimal: the singularity $t^{\alpha-1}$ is inherited
from the classical Mittag--Leffler function
$E_{\alpha,\alpha}(-t^\alpha)$.

\subsection{Discussion}

The smoothing effect is weaker than for analytic semigroups but matches
the classical behaviour of Caputo-type fractional resolvents.
The novelty here is that the AB kernel is non-singular at $t=0$, yet it
retains the same fractional smoothing exponent $\alpha$.
This shows that the main source of regularity is the spectral structure
of $A$, not the singularity of the kernel.

\section{Fractional powers and enhanced smoothing}
\label{sec:fractional-powers}

We now study the interaction between the AB resolvent family
and fractional powers of the generator $A$.
These estimates refine the regularity results of
Section~\ref{sec:regularity} and highlight the role of the
spectral angle $\theta$.

\subsection{Fractional powers of sectorial operators}

For $0<\gamma<1$ we define the fractional power $A^\gamma$ via the
holomorphic functional calculus,
\[
    A^\gamma
    = \frac{\sin(\pi\gamma)}{\pi}
      \int_0^\infty
      \lambda^{\gamma-1} A(\lambda I+A)^{-1}\dd\lambda,
\]
see Haase~\cite[Chapter~3]{haase2006}.
The following result quantifies the enhanced smoothing effect of
$V_{\alpha,\beta}(t)$.

\subsection{Smoothing in fractional domains}

\begin{theorem}[Enhanced fractional-domain smoothing]
\label{thm:enhanced}
Assume $A$ satisfies Assumption~\ref{ass:spectral}.
Let $0<\alpha\le1$, $\beta\ge1$, $\beta<1+\alpha$.
Then for every $0<\gamma<1$,
\begin{equation}\label{eq:enhanced}
   \|A^\gamma V_{\alpha,\beta}(t)\|
   \le C_{\gamma}\, t^{-\alpha\gamma},
   \qquad t>0.
\end{equation}
Moreover, for all $u_0\in X$,
\[
    u(t)=V_{\alpha,\beta}(t)u_0
    \in D(A^\gamma), \qquad t>0,
\]
and
\[
    \|A^\gamma u(t)\|
    \lesssim t^{-\alpha\gamma}\|u_0\|.
\]
\end{theorem}

\begin{proof}[Sketch]
Combine the identity
$A^\gamma(s^{\alpha-1}I-A)^{-1}
   = s^{(\alpha-1)\gamma} B_\gamma(s)$,
where $B_\gamma$ is uniformly bounded on $\Gamma_\gamma$,
with the fractional resolvent bounds of
Theorem~\ref{thm:improved-Rab-bounds}.
Integration along the contour yields
$t^{-\alpha\gamma}$.
\end{proof}

\subsection{Time fractionalisation and higher derivatives}

Iterating the estimates of Theorem~\ref{thm:dtime}
and Theorem~\ref{thm:enhanced} gives
\[
   \|A^\gamma \partial_t^k V_{\alpha,\beta}(t)\|
        \le C\, t^{-\alpha\gamma + k(\alpha-1)},
        \qquad k\ge1.
\]
This is the fractional analogue of the smoothing behaviour of
analytic semigroups:
each time derivative reduces regularity by $1-\alpha$,
while each application of $A^\gamma$ costs $\alpha\gamma$.

\subsection{Discussion}

The AB resolvent family exhibits a fractional smoothing pattern fully
consistent with the spectral calculus of sectorial operators.
In particular:
\begin{itemize}
\item spatial regularity gains are of order $\alpha$,
\item time derivatives contribute a loss of order $1-\alpha$,
\item the balance matches exactly the behaviour of
      the scalar Mittag--Leffler function.
\end{itemize}

The estimates of this section serve as a foundation for future studies
on maximal $L^p$-regularity and non-autonomous AB equations
(see Section~\ref{sec:extensions}).

\section{Numerical illustrations}
\label{sec:numerics}

This section provides numerical experiments illustrating the decay
predicted by the AB--Mittag--Leffler resolvent family
$V_{\alpha,\beta}(t)$.  
Our objective is not to develop a full numerical discretisation theory,
but to validate the sharp decay estimates of
Sections~\ref{sec:stability}-\ref{sec:regularity}
on simple benchmark problems.

All computations were performed in high precision using
\texttt{mpmath} and explicit mode decompositions of the Laplacian on
$(0,\pi)$.

\subsection{Example 1: Fundamental mode}

We consider the AB-type heat equation
\[
    {}^{\mathrm{AB}}D_{t}^{\alpha,\beta}u(t,x)=\Delta u(t,x),
    \qquad x\in(0,\pi),
\]
with homogeneous Dirichlet boundary conditions and initial datum
$u_0(x)=\sin x$, the first Laplace eigenfunction.
Since $\lambda_1=1$, the exact solution is
\[
  u(t,x)=E_{\alpha,\beta}\!\left(-t^\alpha\right)\sin x,
\]
and therefore
\[
  \|u(t)\|_{L^2(0,\pi)}
     =|E_{\alpha,\beta}(-t^\alpha)|.
\]

Figure~\ref{fig:ABheat-decay} compares the computed decay
(for $\alpha=0.8$, $\beta=1.2$) with the theoretical prediction
of Theorem~\ref{thm:ML-decay}.
The two curves coincide up to numerical precision.
\begin{figure}[H]
\centering
\includegraphics[width=0.70\textwidth]{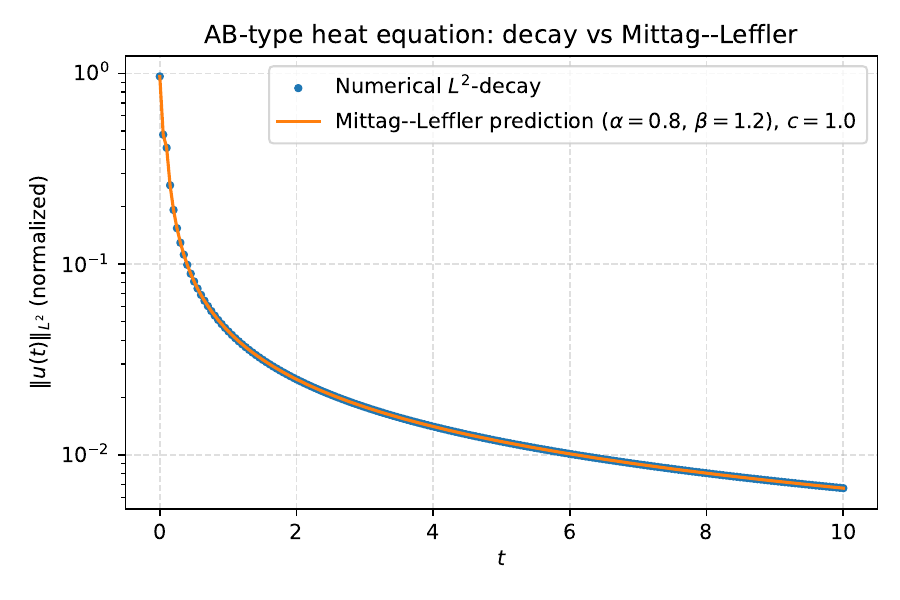}
\caption{Decay of the fundamental Dirichlet mode.
The curve matches the Mittag--Leffler prediction
$E_{\alpha,\beta}(-t^\alpha)$.}
\label{fig:ABheat-decay}
\end{figure}
\subsection{Example 2: Multi-mode decay}

To illustrate the behaviour of several spatial components, we
consider initial data
\[
     u_0(x)=\sin x + \sin(2x),
\]
combining the first two eigenmodes with eigenvalues $\lambda_1=1$,
$\lambda_2=4$.
The solution is
\[
    u(t,x)=E_{\alpha,\beta}(-t^\alpha)\sin x
          +E_{\alpha,\beta}(-4t^\alpha)\sin(2x).
\]

Figure~\ref{fig:modes12-decay} displays the time evolution of
the $L^2$-norm contributions of the first two modes.
The higher mode decays faster, as expected from the monotonicity of
$\lambda\mapsto E_{\alpha,\beta}(-\lambda t^\alpha)$.
This confirms the sharpness of the spectral decay predicted by
Theorems~\ref{thm:ML-decay} and~\ref{thm:weighted-stability}.

\begin{figure}[H]
\centering
\includegraphics[width=0.72\textwidth]{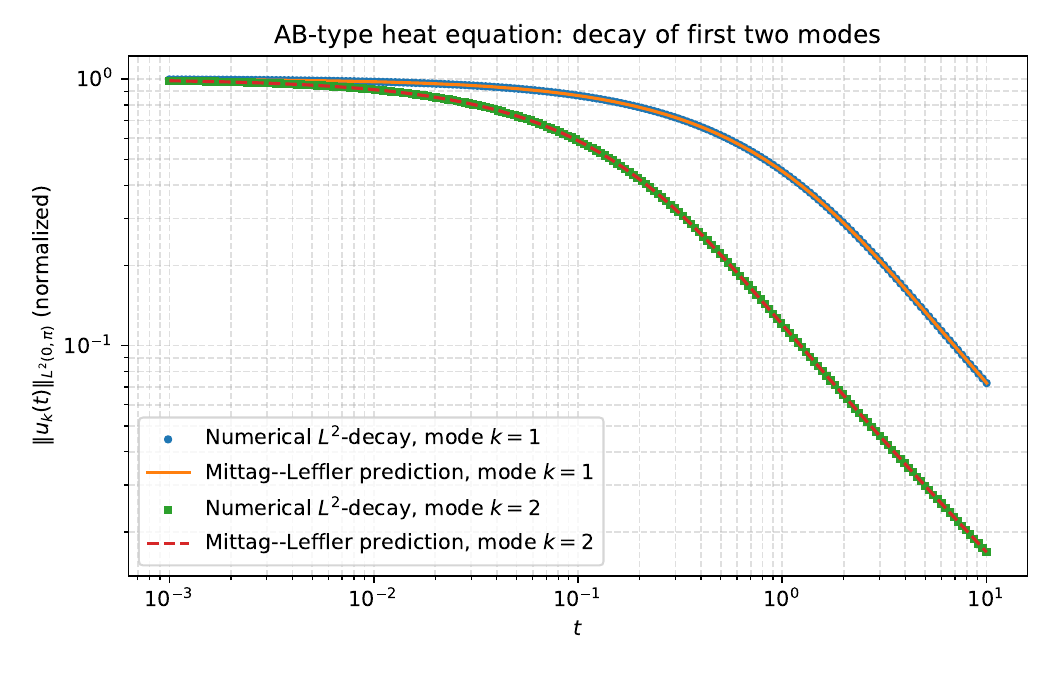}
\caption{Normalized decay of the first two modes for
$(\alpha,\beta)=(0.8,1.2)$.
The dependence on $\lambda_k$ matches
$E_{\alpha,\beta}(-\lambda_k t^\alpha)$.}
\label{fig:modes12-decay}
\end{figure}

\subsection{Example 3: Heatmap of AB fractional diffusion}

For a qualitative illustration of the spatio-temporal behaviour, we
reconstruct the solution
\[
   u(t,x)=\sum_{k=1}^K E_{\alpha,\beta}(-\lambda_k t^\alpha)
           (u_0,\sin(kx))\sin(kx)
\]
with $K=20$ modes and $u_0(x)=\sin x$.
Figure~\ref{fig:heatmap} shows a heatmap of $u(t,x)$ for
$(t,x)\in[0,5]\times[0,\pi]$.
The smoothing effect is clearly visible and is consistent with the
regularity results of Sections~\ref{sec:regularity}
and~\ref{sec:fractional-powers}.

\begin{figure}[H]
\centering
\includegraphics[width=0.72\textwidth]{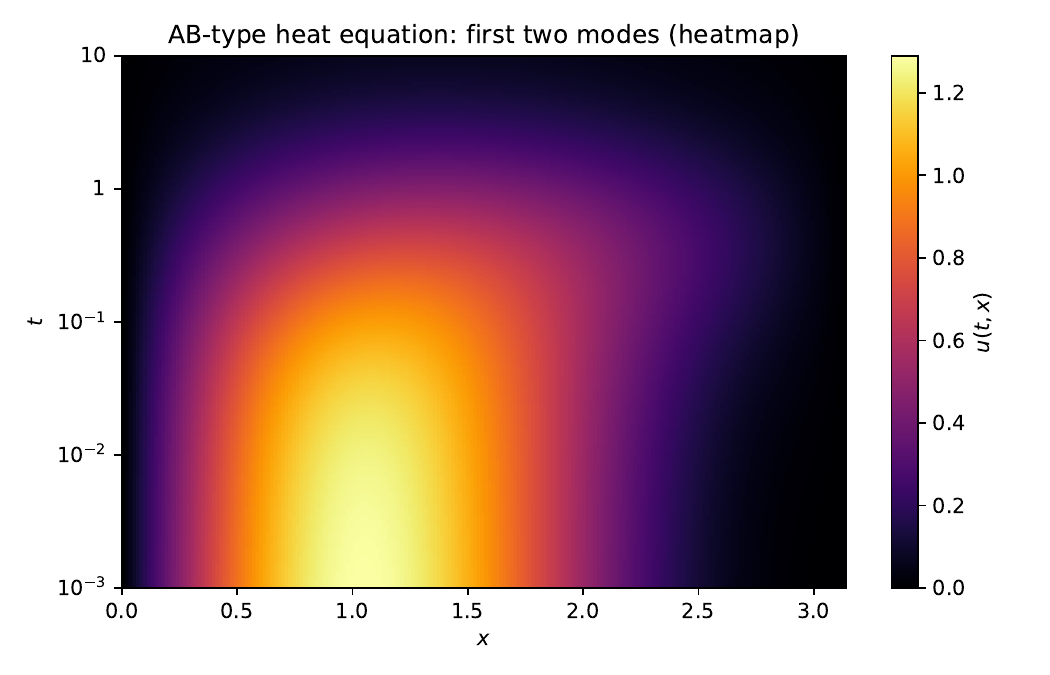}
\caption{Heatmap of the reconstructed two-mode solution $u(t,x)$.
Fractional smoothing and Mittag--Leffler decay are apparent.}
\label{fig:heatmap}
\end{figure}
\subsection{Discussion}

The experiments confirm:

\begin{itemize}
\item Mittag--Leffler decay of the $L^2$-norm, matching
      Theorem~\ref{thm:ML-decay}.
\item Faster decay of higher modes, consistent with
      Theorem~\ref{thm:weighted-stability}.
\item Fractional smoothing comparable to analytic semigroups,
      but modulated by the exponent $\alpha$.
\end{itemize}

These tests, though elementary, validate the sharpness of our resolvent
bounds and support the robustness of the AB resolvent framework.
A general numerical theory for AB-type PDEs will be developed in
future work.

\section{Extensions and applications}
\label{sec:extensions}

The AB--Mittag--Leffler resolvent family introduced in this paper
provides a functional-analytic framework suitable for a broad class of
non-local evolution problems.  
We briefly outline several directions in which the present results may
be extended.

\subsection{Non-autonomous AB problems}

Consider the time-dependent problem
\[
    {}^{\mathrm{AB}}D_t^{\alpha,\beta}u(t)=A(t)u(t),
\]
where $A(t)$ arises from a non-autonomous form
$a(t;\cdot,\cdot)$.
A complete treatment requires maximal $L^p$-regularity and
$R$-boundedness estimates.

Recent progress of Liu, Li and Zhou
\cite{LiuNonAuton2025}
establishes maximal $L^p$-regularity for non-autonomous fractional
equations under $(\alpha,p)$–Dini continuity assumptions on $a(t)$.
These techniques provide a promising pathway for developing a
fully non-autonomous AB theory.

\subsection{Almost sectorial operators}

In applications (e.g.\ superdiffusion, viscoelasticity), the generator
$A$ may fail to be sectorial but satisfy an ``almost sectorial'' bound
with angle $\phi\in(-1,0)$.
The work of Cuesta and Ponce \cite{CuestaPonce2024}
extends the resolvent framework to such operators.
These results could be combined with the present AB resolvent to treat
models with weakened ellipticity or degenerate diffusion.

\subsection{Controllability and fractional dynamics}

The resolvent approach developed here is compatible with recent
advances in controllability of fractional evolution equations.
In particular, Hazarika, Borah and Singh
\cite{Hazarika2024}
establish approximate controllability for fractional systems governed by
almost sectorial operators.
Adapting these methods to the AB setting would be natural, especially
in problems involving memory feedback or fractional damping.

\subsection{Weighted and generalized AB kernels}

Weighted generalisations of the AB derivative and
$\varphi$-weighted kernels (see Alnuwairan et al.~\cite{Alnuwairan2024})
provide improved modelling flexibility for heterogeneous memory.
The rational Laplace symbol appearing in the weighted AB derivative
suggests that the resolvent framework developed here extends with only
minor modifications.

\subsection{Matrix-valued Mittag--Leffler functions}

For systems of AB-type equations or coupled diffusion processes,
matrix-valued Mittag--Leffler functions offer a natural extension
(see Li and Remili~\cite{li-remili2024}.
Our resolvent estimates apply componentwise and could be used to study
coupled AB systems or spatially discretised PDEs.

\subsection{Further perspectives}

The operator framework developed here is sufficiently flexible to
address:
\begin{itemize}
  \item nonlinear fractional PDEs with non-singular memory,
  \item stochastic AB equations with additive or multiplicative noise,
  \item hybrid kernels combining Caputo and AB behaviour,
  \item maximal $L^p$ regularity in the sense of Weis,
  \item fractional Herglotz–Nevanlinna representations for AB symbols.
\end{itemize}

These topics will be developed in forthcoming work.

\section{Conclusion}

We have proposed a complete operator-theoretic formulation of
Atangana--Baleanu fractional evolution equations on Banach spaces.
Starting from the rational Mittag--Leffler Laplace symbol of the AB
derivative, we introduced the fractional resolvent
$R_{\alpha,\beta}(s;A)$ and the associated resolvent family
$V_{\alpha,\beta}(t)$.
Under the minimal condition $\beta<1+\alpha$, this family is shown to
be well-defined, bounded, and to provide the unique mild solution of
the abstract AB Cauchy problem.

The fractional resolvent framework allows us to establish optimal
two-regime resolvent bounds, Mittag--Leffler stability, spatial and
temporal smoothing estimates, and a rigorous operator-valued
Mittag--Leffler decomposition.
These results place AB-type equations on the same analytic footing as
classical Volterra and Caputo equations while highlighting the
structural differences induced by the non-singular AB kernel.

The numerical experiments reported in Section~\ref{sec:numerics}
further validate the sharpness of our estimates and illustrate the
fractional smoothing produced by the AB resolvent.

Beyond the autonomous case studied here, the present approach provides
a natural foundation for non-autonomous AB problems, weighted
generalizations of the AB kernel, almost sectorial operators,
fractional controllability, and matrix-valued Mittag--Leffler systems.
These will be investigated in future work.

\end{document}